%% file: staircases.tex
\documentclass[10pt]{article}

\usepackage{lmodern}
\usepackage[T1]{fontenc}
\usepackage{amsmath}
\usepackage{amsthm}
\usepackage{amssymb}
\usepackage{mathabx} 
\usepackage{url}
\usepackage{latexsym}
\usepackage{titlefoot}
\usepackage[small]{titlesec}
\usepackage{units} 
\usepackage[small,it]{caption}

\usepackage[square,comma,numbers,sort&compress]{natbib}

\usepackage{tikz}
\usetikzlibrary{arrows, automata, decorations.pathreplacing, fit, matrix, patterns, positioning}
\usepackage{tikz-qtree}
\usepackage{xifthen} 

\usepackage{color}
\definecolor{lightgray}{rgb}{0.8, 0.8, 0.8}
\definecolor{darkgray}{rgb}{0.7, 0.7, 0.7}

\usepackage[bookmarks]{hyperref}
\hypersetup{
	colorlinks=true,
	linkcolor=black,
	anchorcolor=black,
	citecolor=black,
	urlcolor=black,
	pdfpagemode=UseThumbs,
	pdftitle={On the Growth of Merges and Staircases of Permutation Classes},
	pdfsubject={Combinatorics},
	pdfauthor={Albert, Pantone, and Vatter},
}

\newcounter{todocounter}

\theoremstyle{plain}
\newtheorem{theorem}{Theorem}
\newtheorem{proposition}[theorem]{Proposition}

\newtheorem{question}[theorem]{Question}

\theoremstyle{definition}

\makeatletter
\newfont{\footsc}{cmcsc10 at 8truept}
\newfont{\footbf}{cmbx10 at 8truept}
\newfont{\footrm}{cmr10 at 10truept}
\renewenvironment{abstract}{
	\begin{list}{}%
	{\setlength{\rightmargin}{1in}%
	\setlength{\leftmargin}{1in}}%
	\item[]\ignorespaces\begin{small}}%
	{\end{small}\unskip\end{list}%
}

\newcommand{\Av}{\operatorname{Av}}
\newcommand{\Grid}{\operatorname{Grid}}

\newcommand{\C}{\mathcal{C}}
\newcommand{\D}{\mathcal{D}}

\newcommand{\M}{\mathcal{M}}

\newcommand{\gr}{\mathrm{gr}}
\newcommand{\lgr}{\underline{\gr}}
\newcommand{\ugr}{\overline{\gr}}

\renewcommand{\vec}[1]{\mathbf{#1}}
\input{pp-doodles}

\datefoot{\today}
\amssubj{05A05, 05A15}

\newpagestyle{main}[\small]{
	\headrule
	\sethead[\usepage][][]
	{\sc On the Growth of Merges and Staircases of Permutation Classes}{}{\usepage}
}

\title{\sc On the Growth of Merges and Staircases of Permutation Classes}

\author{\centering
\begin{tabular}{ccccc}
Michael Albert
&\rule{0pt}{0pt}&
Jay Pantone
&\rule{0pt}{0pt}&
Vincent Vatter\footnote{Vatter's research was partially supported by the National Science Foundation under Grant Number DMS-1301692.}\\[-0.25ex]
\small Department of Computer Science
&&
\small Department of Mathematics
&&
\small Department of Mathematics\\[-0.5ex]
\small University of Otago
&&
\small Dartmouth College
&&
\small University of Florida\\[-0.5ex]
\small Dunedin, New Zealand
&&
\small Hanover, New Hampshire USA
&&
\small Gainesville, Florida USA\\[-1.5ex]
\end{tabular}
}

\titleformat{\section}{\large\sc}{\thesection.}{1em}{}
\date{}
 
\begin{document}
\maketitle

\pagestyle{main}

\begin{abstract}
There is a well-known upper bound on the growth rate of the merge of two permutation classes. Curiously, there is no known merge for which this bound is not achieved. Using staircases of permutation classes, we provide sufficient conditions for this upper bound to be achieved. In particular, our results apply to all merges of principal permutation classes. We end by demonstrating how our techniques can be used to reprove a result of B\'ona.
\end{abstract}

\section{Introduction}

Let $\pi$ be a permutation of length $n$. We utilize one-line notation, writing $\pi = \pi(1)\pi(2)\cdots\pi(n)$. A permutation $\sigma$ of length $k$ is \emph{contained in} $\pi$ (denoted $\sigma \leq \pi$) if there exist indices $1 \leq i_1 < i_2 < \cdots < i_k \leq n$ such that the subsequence $\pi(i_1), \pi(i_2), \ldots, \pi(i_k)$ of $\pi$ is in the same relative order as $\sigma$. For example, the permutation $32514$ contains the permutation $132$ because of the subsequence $254$ (among others). If $\pi$ does not contain $\sigma$, then we say that $\pi$ \emph{avoids} $\sigma$.

A set of permutations $\C$ is called a \emph{permutation class} if it has the property that when $\pi \in \C$ and $\sigma \leq \pi$, then $\sigma \in \C$. Permutation classes are precisely the downsets in the poset of all permutations induced by the containment relation. 

Let $\C_n$ denote the set of permutations in $\C$ of length $n$. We define the \emph{lower and upper (exponential) growth rates} by
\[
	\lgr(\C) = \liminf_{n\rightarrow\infty} \sqrt[n]{|\C_n|},\quad \text{ and } \quad \ugr(\C)=\limsup_{n\rightarrow\infty} \sqrt[n]{|\C_n|},
\]
respectively. When these quantities are equal, their common value is called the \emph{(proper) growth rate} of $\C$ and is denoted $\gr(\C)$. The Marcus--Tardos Theorem~\cite{marcus:excluded-permut:} asserts that the lower and upper growth rates of all classes other than the class of all permutations are finite. It is conjectured that they are always equal, i.e., that every class has a proper growth rate.

Every permutation class is uniquely defined by the set $B$ of minimal permutations avoided by all permutations in the class, called its \emph{basis}. The class with basis $B$ is denoted $\Av(B)$. There has been particular interest in computing the growth rates of \emph{principal classes}, which are those defined by avoiding a single permutation, i.e., those with a singleton basis.

Given two permutation classes $\C$ and $\D$, their \emph{merge}, written $\C\odot\D$, is the set of all permutations whose entries can be colored red and blue so that the red subsequence is order isomorphic to a member of $\C$ and the blue subsequence is order isomorphic to a member of $\D$. We further call such a coloring a \emph{$(\C,\D)$ coloring}.

Merges seem to appear only rarely ``in nature'', with two notable exceptions. First, it is well known that the entries of a $k\cdots 21$-avoiding permutation can be partitioned into $k-1$ increasing subsequences, from which it follows that
\[
	\Av(k\cdots 21)
	=
	\underbrace{\Av(21)\odot\cdots\odot\Av(21)}_{\text{$k-1$ copies of $\Av(21)$}}
	=
	\Av((k-1)\cdots 21)\odot\Av(21).
\]
First studied by Stankova~\cite{stankova:forbidden-subse:}, the class of \emph{skew merged permutations} is also a merge; it is the class $\Av(21)\odot\Av(12)$.

Very little is known about the asymptotic behavior of the sequence $|(\C \odot \D)_n|$, even when the asymptotic behaviors of $|\C_n|$ and $|\D_n|$ are known exactly. An upper bound on $|(\C\odot\D)_n|$ can be obtained by noting that there are ${n\choose i}^2$ ways to partition a permutation of length $n$ into two subpermutations of lengths $i$ and $n-i$, yielding 
\begin{equation}
\label{eqn-merge-upper-bound}\tag{$\dagger$}
	|(\C\odot\D)_n|
	\le
	\sum_{i=0}^n {n\choose i}^2|\C_i||\D_{n-i}|.
\end{equation}
A comparison between \eqref{eqn-merge-upper-bound} and the Binomial Theorem yields the following upper bound on $\ugr(\C\odot\D)$, which first appeared implicitly in the work of the first author~\cite{albert:on-the-length-o:} and was rediscovered by Claesson, Jel{\'{\i}}nek, and Steingr{\'{\i}}msson~\cite{claesson:upper-bounds-fo:}).

\begin{proposition}
\label{prop-merge-gr}
For any two permutation classes $\C$ and $\D$,
\[
	\ugr(\C\odot\D)
	\le
	\left(\sqrt{\ugr(\C)}+\sqrt{\ugr(\D)}\right)^2.
\]
\end{proposition}

The only known lower bound on the growth rate of $\C\odot\D$ in general is $\gr(\C)+\gr(\D)$, which is achieved by their \emph{juxtaposition}, the class of permutations which consist of a prefix order isomorphic to a member of $\C$ followed by a suffix order isomorphic to a member of $\D$. Despite the large gap between bounds, we are not aware of any pair of permutation classes whose merge does \emph{not} achieve the bound in Proposition~\ref{prop-merge-gr}. This is the main question addressed here: Under what conditions on $\C$ and $\D$ can we guarantee that the upper bound on the growth rate of $\C \odot \D$ provided by Proposition~\ref{prop-merge-gr} is actually achieved?

One sufficient condition for the growth rate of $\C\odot\D$ to achieve the upper bound in Proposition~\ref{prop-merge-gr} is that $\C$ and $\D$ have finite intersection:

\begin{proposition}[Albert~\cite{albert:on-the-length-o:}]
\label{prop-merge-gr-finite-intersection}
For any two permutation classes $\C$ and $\D$ that have proper growth rates and finite intersection,
\[
	\gr(\C\odot\D)
	=
	\left(\sqrt{\gr(\C)}+\sqrt{\gr(\D)}\right)^2.
\]
\end{proposition}

To establish Proposition~\ref{prop-merge-gr-finite-intersection}, suppose that the longest permutation in $\C\cap\D$ has length $m$. For each permutation $\pi\in\C\odot\D$ fix a distinguished $(\C,\D)$ coloring. Now take any other $(\C,\D)$ coloring, and define a vector $\vec{d}\in\{0,1\}^n$ by setting $\vec{d}_i=1$ if $\pi(i)$ is a different color in this coloring than it is in the distinguished coloring. The subsequence of $\pi$ consisting of those entries which are red in the distinguished coloring but blue in this coloring is order isomorphic to a member of $\C\cap\D$, and thus has length at most $m$. Similarly, there can be at most $m$ entries which are blue in the distinguished coloring but red in this coloring. Therefore $\vec{d}$ can have at most $2m$ entries equal to $1$, giving us the inequality
\[
	\sum_{i=0}^n {n\choose i}^2|\C_i||\D_{n-i}|
	\le
	|(\C\odot\D)_n|\sum_{i=0}^{2m} {n\choose i}.
\]
The sum on the right-hand side is a degree $2m+1$ polynomial in $n$ and thus does not affect the $n^{\mbox{\scriptsize th}}$ roots in the limit, which together with the upper bound given by Proposition~\ref{prop-merge-gr} proves Proposition~\ref{prop-merge-gr-finite-intersection}.

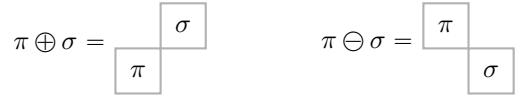
\begin{figure}
\begin{center}
	$\pi\oplus\sigma=$
	\begin{tikzpicture}[scale=0.3, baseline=(current bounding box.center)]
		\plotpermbox{0}{0}{1}{1};
		\plotpermbox{2}{2}{3}{3};
		\node at (0.5,0.5) {$\pi$};
		\node at (2.5,2.5) {$\sigma$};
	\end{tikzpicture}
\quad\quad\quad\quad
	$\pi\ominus\sigma=$
	\begin{tikzpicture}[scale=0.3, baseline=(current bounding box.center)]
		\plotpermbox{0}{2}{1}{3};
		\plotpermbox{2}{0}{3}{1};
		\node at (0.5,2.5) {$\pi$};
		\node at (2.5,0.5) {$\sigma$};
	\end{tikzpicture}
\end{center}
\caption{The sum and skew sum operations.}
\label{fig-sums}
\end{figure}

In the next section we establish another sufficient condition for the growth rate of $\C\odot\D$ to match the upper bound in Proposition~\ref{prop-merge-gr}. To state it, we must first introduce two operations that can be performed on permutations. The \emph{sum}, $\pi\oplus\sigma$, of the permutations $\pi$ and $\sigma$ is the permutation consisting of $\pi$ followed by a shifted copy of $\sigma$, as shown on the left of Figure~\ref{fig-sums}. The drawing on the right of that figure shows a symmetry of this operation called the \emph{skew sum}. A class $\C$ is \emph{sum closed} if $\pi\oplus\sigma\in\C$ for all $\pi,\sigma\in\C$ and \emph{skew closed} if $\pi\ominus\sigma\in\C$ for all $\pi,\sigma\in\C$.

A simple application of Fekete's Lemma for super-multiplicative sequences shows that sum closed and skew closed classes have proper growth rates (this argument was first given by Arratia~\cite{arratia:on-the-stanley-:}). We can now state our main result.

\begin{theorem}
\label{thm-merge-gr-staircase}
If each of the classes $\C$ and $\D$ is either sum or skew closed then
\[
	\gr(\C\odot\D)
	=
	\left(\sqrt{\gr(\C)}+\sqrt{\gr(\D)}\right)^2.
\]
\end{theorem}

In particular, all principal classes are either sum  or skew closed, and thus we see that the growth rate of the merge of any two principal classes is equal to the upper bound in Proposition~\ref{prop-merge-gr}.

A striking example of the usefulness of Theorem~\ref{thm-merge-gr-staircase} is its application to $\Av(k\cdots 21)$. Because permutations in this class can be partitioned into $k-1$ increasing subsequences, it is easy---even without appealing to Proposition~\ref{prop-merge-gr}---to see that $|\Av_n(k\cdots 21)|\le (k-1)^{2n}$, and thus that $\gr(\Av(k\cdots 21))\le (k-1)^2$. That this upper bound is the actual growth rate was first established by Regev~\cite{regev:asymptotic-valu:} via a deep argument (though it should be noted that Regev established quite a bit more than establishing these growth rates). However this fact follows easily from Theorem~\ref{thm-merge-gr-staircase} via induction because $\Av(k\cdots 21)=\Av((k-1)\cdots 21)\odot\Av(21)$.

With no known counterexamples, we are compelled to ask if the upper bound on the growth rate of the merge of two classes is always correct:

\begin{question}
\label{question-merge-gr}
Is it the case that
\[
	\gr(\C\odot\D)
	=
	\left(\sqrt{\gr(\C)}+\sqrt{\gr(\D)}\right)^2
\]
for every pair of classes $\C$ and $\D$ with proper growth rates?
\end{question}

We prove Theorem~\ref{thm-merge-gr-staircase} in the next section and present an application of it in Section~\ref{sec-principal-gr}. We conclude by discussing a candidate for the ``next'' most obvious merge to consider in investigating Question~\ref{question-merge-gr}.

\section{Staircases}

In order to prove Theorem~\ref{thm-merge-gr-staircase} we must take a detour to study certain permutation classes that we call \emph{staircase classes}, which are special cases of infinite grid classes of permutations. Therefore, in this section we will first define grid classes and recall an important result about their growth rates. Then, we will define staircase classes, compute a bound on their growth rates, and conclude the section by illustrating the connection between staircase classes and merges.

Suppose that $\M$ is a $t\times u$ matrix of permutation classes, where $t$ is the number of columns, and $u$ the number of rows. An \emph{$\M$-gridding} of the permutation $\pi$ of length $n$ is a pair of sequences $1=c_1\le\cdots\le c_{t+1}=n+1$ (the column divisions) and $1=r_1\le\cdots\le r_{u+1}=n+1$ (the row divisions) such that for all $1\le k\le t$ and $1\le\ell\le u$, the entries of $\pi$ with indices in $[c_k,c_{k+1})$ and values in $[r_{\ell}, r_{\ell+1})$ are order isomorphic to an element of $\M_{k,\ell}$.  The \emph{grid class of $\M$}, written $\Grid(\M)$, consists of all permutations which possess an $\M$-gridding. The aforementioned juxtaposition of $\C$ and $\D$ can be expressed in this language as $\Grid(\C\ \D)$

By relating their growth rates to the asymptotics of certain walks in a bipartite graph, Bevan~\cite{bevan:growth-rates-of:} gave a formula for growth rates of \emph{monotone grid classes}, that is, those where every cell is either the empty class $\emptyset$, the increasing class $\Av(21)$, or the decreasing class $\Av(12)$. The first and third authors have since established the following generalization.

\begin{theorem}
[Albert and Vatter~\cite{albert:an-elementary-p:}]
\label{thm-bevan-grid-gr}
Let $\M$ be a $t\times u$ matrix of permutation classes, each with a proper growth rate, and define the $t\times u$ matrix $\Gamma$ by $\Gamma_{k,\ell}=\sqrt{\gr(\M_{k,\ell})}$. The growth rate of $\Grid(\M)$ is equal to the greatest eigenvalue of $\Gamma^T\Gamma$ \textup{(}or equivalently, of $\Gamma\Gamma^T$\textup{)}.
\end{theorem}

\begin{figure}
\[
	\begin{array}{ccc}
		\begin{tikzpicture}[scale=0.3, baseline=(current bounding box.center)]
			\foreach \i in {1,3,5,7,9} {
				\plotpermbox{\i}{\i}{\i+1}{\i+1};
				\node at (\i+0.5,\i+0.5) {$\C$};
				\plotpermbox{\i+2}{\i}{\i+3}{\i+1};
				\node at (\i+2.5,\i+0.5) {$\D$};
			}
			\draw[white] (-0.75,-0.75) rectangle (11.75,11.75);
			\useasboundingbox (current bounding box.south west) rectangle (current bounding box.north east);
			\node[rotate=45] at (11.35,11.35) {$\dots$};
			\node[rotate=45] at (13.35,11.35) {$\dots$};
		\end{tikzpicture}
	&\quad\quad&
		\begin{tikzpicture}[scale=0.3, baseline=(current bounding box.center)]
			\foreach \i in {1,3,5,7,9} {
				\plotpermbox{\i}{\i}{\i+1}{\i+1};
				\node at (\i+0.5,\i+0.5) {$\C$};
			}
			\plotpermbox{1}{9}{2}{10};
			\node at (1.5,9.5) {$\D$};
			\plotpermbox{3}{7}{4}{8};
			\node at (3.5,7.5) {$\D$};
			\plotpermbox{7}{5}{8}{6};
			\node at (7.5,5.5) {$\D$};
			\plotpermbox{9}{3}{10}{4};
			\node at (9.5,3.5) {$\D$};
			\draw[white] (-0.75,-0.75) rectangle (11.75,11.75);
			\useasboundingbox (current bounding box.south west) rectangle (current bounding box.north east);
			\node[rotate=45] at (-0.25,-0.25) {$\dots$};
			\node[rotate=135] at (-0.35,11.35) {$\dots$};
			\node[rotate=135] at (11.25,1.75) {$\dots$};
			\node[rotate=45] at (11.35,11.35) {$\dots$};
		\end{tikzpicture}
	\end{array}
\]
\caption{The two staircases we use: the infinite increasing $(\C,\D)$ staircase on the left and the infinite counterclockwise spiral $(\C,\D)$ staircase on the right.}
\label{fig-CD-staircases}
\end{figure}
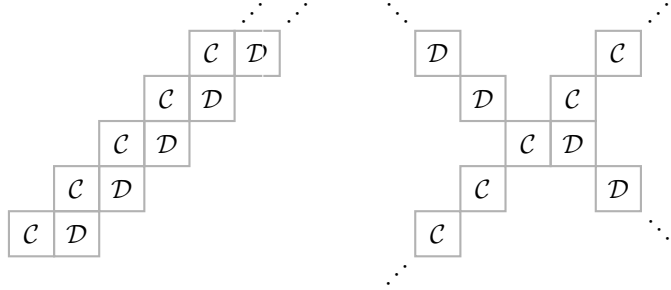

A picture of the infinite increasing $(\C,\D)$ staircase is shown on the left of Figure~\ref{fig-CD-staircases}. Before defining this staircase as a grid class, we should warn the reader that, so that the entries of our matrices align with those of our permutations, we index matrices in Cartesian coordinates. Thus $\M_{k,\ell}$ denotes the entry in the $k^{\mbox{\scriptsize th}}$ row from the bottom and the $\ell^{\mbox{\scriptsize th}}$ column from the left. With that warning issued, the infinite increasing $(\C,\D)$ staircase is equal to
\[
	\Grid\left(\begin{array}{cccc}
	&&\reflectbox{$\ddots$}&\reflectbox{$\ddots$}\\
	&\C&\D\\
	\C&\D\\
	\end{array}\right).
\]
In our indexing, the entries of the main diagonal of the matrix are equal to $\C$ and the entries of the adjacent diagonal are equal to $\D$.

For the rest of this section we assume that the classes $\C$ and $\D$ both have proper growth rates. We define the \emph{$t$-step} increasing $(\C,\D)$ staircase to be the subclass of the infinite staircase corresponding to the first $t$ rows. The matrix defining the $t$-step increasing $(\C,\D)$ staircase therefore has $t$ rows and $t+1$ columns. For this grid class, the matrix $\Gamma$ of Theorem~\ref{thm-bevan-grid-gr} contains diagonal entries equal to $\sqrt{\gr(\C)}$ and subdiagonal entries equal to $\sqrt{\gr(\D)}$. Furthermore, recalling our unusual matrix indexing, we see that $\Gamma\Gamma^T$ is the $t\times t$ matrix defined by
\[
	(\Gamma\Gamma^T)_{k,\ell}
	=
	\left\{\begin{array}{ll}
	\gr(\C)+\gr(\D)&\text{if $k=\ell$,}\\[4pt]
	\sqrt{\gr(\C)\gr(\D)}&\text{if $|k-\ell|=1$, and}\\[4pt]
	0&\text{otherwise.}
	\end{array}\right.
\]

Thus $\Gamma\Gamma^T$ is a tridiagonal Toeplitz matrix, meaning that its nonzero entries are confined to the main diagonal and the two diagonals immediately above and below it (the \emph{tridiagonal} condition) and that its entries along a given diagonal are identical (the \emph{Toeplitz} condition). Tridiagonal Toeplitz matrices are one of the few families of matrices for which exact formulas for their eigenvalues and eigenvectors are known (for example see Meyer~\cite[Example 7.2.5]{meyer:matrix-analysis:}); the eigenvalues of a $t\times t$ tridiagonal Toeplitz matrix with subdiagonal entries $a$, main diagonal entries $b$, and superdiagonal entries $c$ are given by
\[
	\lambda_j=b+2\sqrt{ac}\cos\left(\frac{j\pi}{t+1}\right)
\]
for $j=1,\dots,t$. Applying this to $\Gamma\Gamma^T$, Theorem~\ref{thm-bevan-grid-gr} implies that the growth rate of any $t$-step $(\C,\D)$ staircase is
\begin{equation}
\label{eqn-t-step-staircase-gr}\tag{$\ddagger$}
	\gr(\C)+2\sqrt{\gr(\C)\gr(\D)}\cos\left(\frac{1}{t+1}\right)+\gr(\D).
\end{equation}
As $t\rightarrow\infty$, the central term approaches $2\sqrt{\gr(\C)\gr(\D)}$, showing that the growth rate of the infinite increasing $(\C,\D)$ staircase is at least $(\sqrt{\gr(\C)}+\sqrt{\gr(\D)})^2$.

Indeed, this is a lower bound on the growth rates of a wide variety of other classes that we also call staircase classes. Let $\M$ be an infinite matrix of permutation classes whose nonempty cells are equal to either $\C$ or $\D$ and label these nonempty cells by the positive integers. We say that $\Grid(\M)$ is an \emph{infinite $(\C,\D)$ staircase} if
\begin{itemize}
\item the first cell is equal to $\C$ and is the only nonempty cell in its column,
\item for all $i\ge 1$, the $2i^{\mbox{\scriptsize th}}$ cell is equal to $\D$, lies in the same row as the $(2i-1)^{\mbox{\scriptsize st}}$ cell, and these are the only two nonempty cells in this row, and
\item for all $i\ge 1$, the $(2i+1)^{\mbox{\scriptsize st}}$ cell is equal to $\C$, lies in the same column as the $2i^{\mbox{\scriptsize th}}$ cell, and these are the only two nonempty cells in this column.
\end{itemize}
Given an infinite $(\C,\D)$ staircase defined by the matrix $\M$, we define its $t$-step restriction to be the grid class formed by the first $2t$ cells of $\M$ (and thus by our rules above, this restriction contains $t$ cells equal to $\C$ and $t$ equal to $\D$).

No matter what infinite $(\C,\D)$ staircase we take, the growth rates of its $t$-step restrictions are given by \eqref{eqn-t-step-staircase-gr}. This can be seen algebraically because row and column permutations of $\Gamma$ do not affect the eigenvalues of $\Gamma\Gamma^T$, or combinatorially because rearranging the rows and columns of the matrix $\M$ does not affect the growth rate of $\Grid(\M)$. By considering limits as $t\rightarrow\infty$, we obtain the following result.

\begin{proposition}
\label{prop-staircase-gr}
The growth rate of any infinite $(\C,\D)$ staircase is at least $\left(\sqrt{\gr(\C)}+\sqrt{\gr(\D)}\right)^2$.
\end{proposition}

We are now ready to establish Theorem~\ref{thm-merge-gr-staircase}. Suppose that each of $\C$ and $\D$ is either sum closed or skew closed. By symmetry, we may suppose that $\C$ is sum closed. If $\D$ is also sum closed, then we see that every member of the infinite increasing $(\C,\D)$ staircase is the merge of a permutation from $\C$ and one from $\D$. As the growth rate of this staircase matches the upper bound on the growth rate of $\C\odot\D$ from Proposition~\ref{prop-staircase-gr}, we are done. Otherwise, $\D$ must be skew closed. In this case, the members of the infinite counterclockwise spiral $(\C,\D)$ staircase shown on the right of Figure~\ref{fig-CD-staircases} are contained in $\C\odot\D$, and again we have achieved the upper bound from Proposition~\ref{prop-staircase-gr}, completing the proof of Theorem~\ref{thm-merge-gr-staircase}.

In some cases, the merge upper bound shows that the lower bound in Proposition~\ref{prop-staircase-gr} is the true growth rate of a staircase. For example, if both $\C$ and $\D$ are sum closed, then the infinite increasing $(\C,\D)$ staircase is contained in $\C\odot\D$, so we know its growth rate. However, unlike the case of merges, we know that the lower bounds in Proposition~\ref{prop-merge-gr-finite-intersection} are \emph{not} always the correct growth rates. 

One rather trivial case which demonstrates this is when we take $\D=\emptyset$; the infinite increasing $(\C,\emptyset)$ staircase is simply the sum closure of $\C$ (the smallest sum closed class containing $\C$), and the lower bound from Proposition~\ref{prop-staircase-gr} shows that the growth rate of this staircase must be at least the growth rate of $\C$ itself. However, taking $\C$ to be the decreasing permutations $\Av(12)$ we see that it has growth rate $1$ while its sum closure (the class of \emph{layered permutations}) has growth $2$.

For another example of inequality where neither class is empty, we turn to the infinite increasing $(\Av(21),\{1\})$ staircase, where again the lower bound in Proposition~\ref{prop-staircase-gr} is $1$. On the other hand, the infinite increasing $(\Av(21),\{1\})$ staircase is easily seen to be the class $\Av(321,4123)$, and it can be shown in a variety of ways (for example, via the insertion encoding of Albert, Linton, and Ru\v{s}kuc~\cite{albert:the-insertion-e:}) that the generating function of this class is $(1-2x)/(1-3x+x^2)$. From this it follows that the growth rate of the infinite increasing $(\Av(21),\{1\})$ staircase is $1+\varphi$ where $\varphi$ denotes the golden ratio.

Thus we might try insisting that both $\C$ and $\D$ are infinite permutation classes. We do not have a proof that equality is not achieved in these cases, but numerical evidence suggests that the growth rate of the infinite increasing $(\Av(12),\Av(12))$ staircase is greater than $4$. This is surprising because the infinite increasing $(\Av(21),\Av(21))$ staircase consists precisely of the permutations in $\Av(321)$, which has a growth rate of $4$.

\section{Growth Rates of Principal Classes}
\label{sec-principal-gr}

It is somewhat remarkable that the crude bounding below by the containment of staircases and above by the naive merge bound can, in some cases, establish the exact growth rates of classes which are not themselves merges. To give a broad family for which this holds we appeal to a result of Jel{\'{\i}}nek and Valtr, who investigated the question of which classes are contained in the merge of two proper subclasses. Strengthening earlier results of B\'ona~\cite{bona:new-records-in-:} and Claesson, Jel{\'{\i}}nek, and Steingr{\'{\i}}msson~\cite{claesson:upper-bounds-fo:}, they established the following (we state their result in a symmetric, skew sum form).

\begin{proposition}[Jel{\'{\i}}nek and Valtr~\cite{jelinek:splittings-and-:}]
\label{prop-split-1}
For all nonempty permutations $\alpha$, $\beta$, and $\gamma$, we have
\begin{equation}
\label{eqn-split-1}\tag{$\mathsection$}
	\Av(\alpha\ominus\beta\ominus\gamma)
	\subseteq
	\Av(\alpha\ominus\beta) \odot \Av(\beta\ominus\gamma).
\end{equation}
\end{proposition}

We caution the reader that while Theorem~\ref{thm-merge-gr-staircase} gives the growth rate of the class on the right-hand side of \eqref{eqn-split-1}, this is generally not the growth rate of the class on the left-hand side. Consider, for example, the permutation $4231=1\ominus 12\ominus 1$, where the class on the right-hand side of \eqref{eqn-split-1} has growth rate $16$. Indeed, establishing this upper bound of $16$ on the growth rate of $\Av(4231)$ was one of the original motivations of Claesson, Jel{\'{\i}}nek, and Steingr{\'{\i}}msson~\cite{claesson:upper-bounds-fo:}. B\'ona~\cite{bona:a-new-record-fo:,bona:a-new-upper-bou:} has since shown how to further restrict the allowable merges to achieve an upper bound of $13.74$.

However, in the case of $\beta=1$, equality is achieved. Consider the infinite $(\Av(\alpha\ominus 1), \Av(1\ominus\gamma))$ increasing staircase for any permutations $\alpha$ and $\gamma$. Suppose to the contrary that a member $\pi$ of this staircase were to contain $\alpha\ominus 1\ominus\gamma$, and consider the position of the entry participating as the `$1$' between the copies of $\alpha$ and $\gamma$. If this entry were to lie in a cell labeled by $\Av(\alpha\ominus 1)$, then there would have to be a copy of $\alpha$ above and to its left, showing that the cell itself contained $\alpha\ominus 1$, a contradiction. Similarly, such an entry cannot lie in a cell labeled by $\Av(1\ominus\gamma)$, as then it could not contain a copy of $\gamma$ below and to its right. This shows that $\Av(\alpha\ominus 1\ominus\gamma)$ contains the infinite $\Av(\alpha\ominus 1), \Av(1\ominus\gamma))$ increasing staircase, implying the following result of B\'ona.

\begin{theorem}[B\'ona~{\cite[Theorem 4.2]{bona:new-records-in-:}}]
\label{thm-bona-triple-skew}
For all permutations $\alpha$ and $\gamma$,
\[
	\gr(\Av(\alpha\ominus 1\ominus\gamma))=\left(\sqrt{\gr(\Av(\alpha\ominus 1))} + \sqrt{\gr(\Av(1\ominus\gamma))}\right)^2.
\]
\end{theorem}

An interesting special case of Theorem~\ref{thm-bona-triple-skew} is the class $\Av(54213)$, where 
\[
	\gr(\Av(54213))
	=
	\gr(\Av(1\ominus 1\ominus 213))
	=
	\left(1+\sqrt{\gr(4213)}\right)^2.
\]
The permutation $4213$ is a symmetry of $1342$, and in \cite{bona:exact-enumerati:}, B\'ona showed that $\gr(\Av(1342))=8$. Thus the growth rate of $\Av(54213)$ is $(1+\sqrt{8})^2=9+4\sqrt{2}$. This result was established by B\'ona in \cite{bona:the-limit-of-a-:}, before proving Theorem~\ref{thm-bona-triple-skew}, and was the first known non-integral growth rate of a principal class.

Our final application of staircases to establish another result of B\'ona.

\begin{theorem}[B\'ona~{\cite[Theorem 5.5]{bona:the-limit-of-a-:}}]
\label{thm-bona-minus-1}
If $\beta$ is sum indecomposable then
\[
	\gr(\Av(1\ominus\beta))\ge \left(1+\sqrt{\gr(\Av(\beta))}\right)^2.
\]
\end{theorem}
\begin{proof}
Because $\beta$ is sum indecomposable, we see that the infinite increasing $(\Av(21),\Av(\beta))$ staircase is contained in $\Av(1\ominus\beta)$, giving the bound.
\end{proof}

\section{Concluding Remarks}
\label{sec-IImergeI}

In searching for further evidence for, or a counterexample to, Question~\ref{question-merge-gr}, Theorem~\ref{thm-merge-gr-staircase} shows that at least one of the classes must be neither sum nor skew closed, and Proposition~\ref{prop-merge-gr-finite-intersection} shows that the two classes must have infinite intersection. We must choose at least one of the classes so that it has no proper sum or skew closed subclass with the same growth rate, as otherwise we could use a staircase construction with such subclasses to achieve the upper bound. For this reason we rule out classes such as $\Av(21)\ominus\Av(21)$, which is neither sum nor skew closed but contains a sum closed class, $\Av(21)$, of the same growth rate.

One of the simplest examples of a merge not covered by known results or resolved by the preceding remarks is $\Grid(\Av(21)\ \Av(21))$ merged with $\Av(21)$. Here we pose the following instance of Question~\ref{question-merge-gr}.

\begin{question}
\label{question-II-merge-I}
Is $\gr(\Grid(\Av(21)\ \Av(21))\odot \Av(21))=3+2\sqrt{2}$?
\end{question}

Although not relevant to the resolution of Question~\ref{question-II-merge-I}, it is curious that this merge is defined by a finite basis (``most'' merges do not seem to be finitely based), in particular,
\[
	\Grid(\Av(21)\ \Av(21))\odot\Av(21)
	=
	\Av(4321, 321654, 421653, 431652, 521643, 531642).
\]
For one direction of this equality, we see that $4321$ does not lie in $\Grid(\Av(21)\ \Av(21))\odot\Av(21)$, and that this merge is contained in $\Grid(\Av(321)\ \Av(321))$ (obtained by taking the merges of the individual cells of the grid with $\Av(21)$). Since $\Grid(\Av(321)\ \Av(321))$ is simply a juxtaposition of two classes, the results of Atkinson~\cite{atkinson:restricted-perm:} describe how to compute its basis. The claimed basis for the merge is in fact the basis of $\Av(4321)\cap\Grid(\Av(321)\ \Av(321))$. For the other direction, take a $4321$-avoiding permutation $\pi\in\Grid(\Av(321)\ \Av(321))$ and express it as the juxtaposition of two sequences $\pi_L$ and $\pi_R$ which are order isomorphic to $321$-avoiding permutations. If a non-left-to-right maximum of $\pi_L$ were greater than a non-right-to-left minimum of $\pi_R$ then it would follow that $\pi$ contained $4321$, a contradiction. Therefore $\pi$ is the merge of the left-to-right maxima of $\pi_L$ and the right-to-left minima of $\pi_R$, which together are order isomorphic to a member of $\Grid(\Av(21)\ \Av(21))$, with the remaining entries of $\pi_L$ and $\pi_R$, which are themselves increasing.

%
%
%
%
%
%

%

\bibliographystyle{acm}
\bibliography{../../refs}

\end{document}

%% file: pp-doodles.tex
\newcommand\mybullet{\raisebox{-5pt}{\normalsize \ensuremath{\bullet}}}
\newcommand\mycirc{\raisebox{-5pt}{\normalsize \ensuremath{\circ}}}

\makeatletter
\def\absdot{\@ifnextchar[{\@absdotlabel}{\@absdotnolabel}}
	\def\@absdotlabel[#1]#2{%
		\node at #2 {\normalsize \mybullet};
		\node at #2 [below=2pt] {\ensuremath{#1}};
	}
	\def\@absdotnolabel#1{%
		\node at #1 {\normalsize \mybullet};
	}
\def\absdothollow{\@ifnextchar[{\@absdothollowlabel}{\@absdothollownolabel}}
	\def\@absdothollowlabel[#1]#2{%
		\node at #2 {\normalsize \textcolor{white}{\mybullet}};
		\node at #2 {\normalsize \mycirc};
		\node at #2 [below=2pt] {\ensuremath{#1}};
	}
	\def\@absdothollownolabel#1{%
		\node at #1 {\normalsize \textcolor{white}{\mybullet}};
		\node at #1 {\normalsize \mycirc};
	}
\makeatother

\newcommand{\plotperm}[1]{
	\foreach \j [count=\i] in {#1} {
		\absdot{(\i,\j)};
	};
}

\newcommand{\plotpermbox}[4]{
	\draw [darkgray, thick, line cap=round]
		({#1-0.5}, {#2-0.5}) rectangle ({#3+0.5}, {#4+0.5});
}

\newcommand{\plotpermgraph}[1]{
	\foreach \j [count=\i] in {#1} {
		\foreach \b [count=\a] in {#1} {
			\ifthenelse{\a<\i \AND \b>\j}{\draw (\a,\b)--(\i,\j);}{}
		};
	};
	\plotperm{#1};
}

\newcommand{\plotpermdyckpath}[1]{
	\draw[ultra thick, line cap=round] (0.5,0.5)
	\foreach \step in {#1} {
		\ifnum\step=1
			-- ++(0,1)
		\else
			-- ++(1,0)
		\fi
	};
}

\newcommand{\plotdyckpath}[1]{
	\draw[ultra thick, line cap=round] (0.5,0)
	\foreach \step in {#1} {
		\ifnum\step=1
			-- ++(1,1)
		\else
			-- ++(1,-1)
		\fi
	};
}

\newcommand{\arcskinnyplain}[2]{
	\draw[thick] (#1,0) arc (180:0:{(#2-#1)/2});
}

\newcommand{\matchsmall}[1]{
	\begin{tikzpicture}[scale=.1, anchor=base]
		\def\h{0};
		\def\maxh{0};
		\foreach \i/\j in {#1} {
			\pgfmathparse{\j-\i};
			\let\h\pgfmathresult;
			\pgfmathifthenelse{\h>\maxh}{\h}{\maxh};
			\global\let\maxh\pgfmathresult;
		};
		\pgftransformyscale{{4.5/\maxh}};
		\foreach \i/\j in {#1} {
			\arcskinnyplain{\i}{\j};
		};
	\end{tikzpicture}
}

\newcommand{\matchpermsmall}[1]{
	\begin{tikzpicture}[scale=.1, anchor=base]
		\foreach \j [count=\n] in {#1} {};
		\def\h{0};
		\def\maxh{0};
		\foreach \j [count=\i] in {#1} {
			\pgfmathparse{2*\n+1-\j-\i};
			\let\h\pgfmathresult;
			\pgfmathifthenelse{\h>\maxh}{\h}{\maxh};
			\global\let\maxh\pgfmathresult;
		};
		\pgftransformyscale{{4.5/\maxh}};
		\foreach \j [count=\i] in {#1} {
			\arcskinnyplain{\i}{{2*\n+1-\j}};
		};
	\end{tikzpicture}
}